\documentclass[a4paper,11pt]{amsart}

\pdfoutput=1

\usepackage[T1]{fontenc}
\usepackage{amssymb,amsfonts,amsmath,amsthm}
\usepackage{url}
\usepackage{color}
\usepackage{enumerate}
\usepackage{epsfig,graphicx,psfrag}
\usepackage{tikz}
\usepackage{asymptote}
\usepackage{pinlabel}

\usepackage[top=1.2in,bottom=1.4in,left=1.4in,right=1.4in]{geometry}

\input{xy}
\xyoption{all}
\objectmargin={3mm}

\newcounter{notes}
\newcommand{\ignore}[1]{}

\newtheorem{theorem}{Theorem}

\newtheorem{corollary}[theorem]{Corollary}
\newtheorem{lemma}[theorem]{Lemma}

\newtheorem{observation}[theorem]{Observation}

\theoremstyle{definition}

\newtheorem{remark}[theorem]{Remark}

\newtheorem{example}[theorem]{Example}
\newtheorem{question}[theorem]{Question}
\newtheorem*{acknowledge}{Acknowledgements}

\newtheoremstyle{theoremwithref}{}{}{\itshape}{}{\bfseries}{.}{.5em}{#1 #2 #3}
\theoremstyle{theoremwithref}

\newcommand{\R}{\mathbb{R}}
\newcommand{\Z}{\mathbb{Z}}

\newcommand{\F}{\mathbb{F}}
\newcommand{\SL}{\mathrm{SL}}
\newcommand{\PSL}{\mathrm{PSL}_2(\mathbb{R})}
\newcommand{\SO}{\mathrm{SO}}

\newcommand{\Hom}{\mathrm{Hom}}
\newcommand{\Aut}{\mathrm{Aut}}
\newcommand{\Out}{\mathrm{Out}}
\newcommand{\Mod}{\mathrm{Mod}}
\newcommand{\Inn}{\mathrm{Inn}}
\renewcommand{\H}{\mathbb{H}}

\newcommand{\HOS}{\mathrm{Homeo}_+(S^1)}
\newcommand{\HOZ}{\mathrm{Homeo}^\Z(\R)}
\newcommand{\rot}{\mathrm{rot}}
\newcommand{\rotild}{\widetilde{\mathrm{ro}}\mathrm{t}}
\newcommand{\eu}{\mathrm{eu}}
\newcommand{\id}{\mathrm{id}}

\DeclareMathOperator{\Homeo}{Homeo}

\DeclareMathOperator{\Diff}{Diff}

\title{Rigidity of mapping class group actions on $S^1$}

\author{Kathryn Mann}
\address{Department of Mathematics, Brown University, 151 Thayer Street, Providence, RI 02912, USA
}
\email{mann@math.brown.edu}

\author{Maxime Wolff}
\address{Sorbonne Universit\'es, UPMC Univ.\ Paris 06, Institut de Math\'ematiques
de Jussieu-Paris Rive Gauche, UMR 7586, CNRS, Univ. Paris Diderot, Sorbonne
Paris Cit\'e, 75005 Paris, France}
\email{maxime.wolff@imj-prg.fr}

\begin{document}

\maketitle
\numberwithin{theorem}{section}

\begin{abstract}
The mapping class group $\Mod_{g, 1}$ of a surface with one marked point can
be identified with an index two subgroup of $\Aut(\pi_1 \Sigma_g)$.
For a surface of genus $g \geq 2$, we show that any action of $\Mod_{g, 1}$
on the circle is either semi-conjugate to its natural action on the Gromov
boundary of $\pi_1 \Sigma_g$, or factors through a finite cyclic group.
For $g \geq 3$, all finite actions are trivial.
This answers a question of~Farb.

\vspace{0.2cm}

\end{abstract}


\section{Introduction} 

Let $\Sigma_g$ be an oriented surface of genus $g \geq 2$, and let $\Gamma_g$
denote $\pi_1(\Sigma_g)$.
The Gromov boundary of $\Gamma_g$ is a topological circle, on which the
group $\Aut(\Gamma_g)$ of automorphisms of $\Gamma_g$ acts faithfully by homeomorphisms.
Geometrically, this boundary action of $\Aut(\Gamma_g)$ can be seen as follows.
By the Dehn-Nielsen-Baer theorem,
the exact sequence $\Inn(\Gamma_g) \to \Aut(\Gamma_g) \to \Out(\Gamma_g)$
is isomorphic, term by term, to the Birman exact sequence
$\pi_1(\Sigma_g) \to \Mod^{\pm}_{g, 1} \to \Mod^{\pm}_g$ of the extended
mapping class group of a surface with one marked point.
Fixing a hyperbolic metric on $\Sigma_g$, the universal cover
$\widetilde{\Sigma_g}$ can be identified with $\H^2$, which has a natural
compactification to a closed disc.
Let $x \in \H^2$ be a lift of the marked point on $\Sigma_g$.
For $f \in \Homeo(\Sigma_g)$ fixing the marked point, let $\tilde{f}$ denote the unique lift of $f$ to
$\H^2$ that fixes $x$. Using the fact that quasi-geodesics remain bounded
distance from geodesics in negative curvature, one can show that the action
of $\tilde{f}$ on $\H^2$ extends to a homeomorphism of the boundary circle,
which depends only on the isotopy class of $f$.
This procedure gives a well-defined homomorphism
$\Mod^{\pm}_{g,1} \to \Homeo(S^1)$, which under the identification
$\Aut(\Gamma_g) \cong \Mod^{\pm}_{g,1}$, is conjugate to the boundary
action described above.
We call this conjugacy class of action the {\em standard action} of
$\Aut(\Gamma_g)$ on $S^1$.

The mapping class group $\Mod_{g,1}$ consisting of isotopy classes of
orientation-preserving homeomorphisms is an index two subgroup of
$\Mod^{\pm}_{g,1}$; we let $\Aut_+(\Gamma_g)$ denote the corresponding
subgroup of $\Aut(\Gamma_g)$.
In~\cite[Question~6.2]{Farb}, B.~Farb asked whether any faithful action of
$\Aut_+(\Gamma_g)$ on $S^1$ by homeomorphisms was necessarily conjugate to the
standard action described above.
In fact one needs to be a little more careful with the statement -- rather
than conjugacy, the appropriate notion of equivalence for $C^0$ actions on
$S^1$ is Ghys' {\em semi-conjugacy}, described below, since 
any action of an infinite group on $S^1$ can be modified (for instance, using
the classical Denjoy trick) to produce non-conjugate, but semi-conjugate
examples.  
Here we give a positive answer to this version of Farb's question.
\begin{theorem} \label{thm:main}
  Let $\rho\colon\Aut_+(\Gamma_g) \to \Homeo_+(S^1)$ be a homomorphism.
  Up to reversing the orientation of the circle, we have the following.
  If $g \geq 3$, then $\rho$ is either trivial or semi-conjugate
  to the standard Gromov boundary action.
  If $g=2$, then $\rho$ is either conjugate to a subgroup of $\Z/10\Z$
  acting by rotations, or is semi-conjugate to the standard action.
\end{theorem}
Actions conjugate to finite groups of rotations do indeed arise in the
genus~2 case. As shown by Mumford~\cite{Mumford}, the abelianization of
$\Aut_+(\Sigma_2)$ is $\Z/10\Z$ (Mumford discusses $\Mod_2$, but the same
is true of $\Mod_{2,1}$, which is also generated by Dehn twists in simple
closed curves. See \cite{Korkmaz}.)
Finite cyclic groups do act on $S^1$ -- necessarily by an action conjugate
to one by rigid rotations. Our theorem simply states that any
non-standard action factors through the abelianization.
Note also that the theorem immediately gives a corresponding statement for
homomorphisms of the larger group $\Aut(\Gamma_g)$ into $\Homeo(S^1)$.

As asserted above, semi-conjugacy is also a necessary hypothesis.  
Following usage of Ghys, two actions $\rho_1$ and $\rho_2: \Gamma \to \HOS$ are said to be 
``semi-conjugate'' if there exists an equivariant, cyclic order
preserving bijection from some orbit of $S^1$ under $\rho_1$ to some orbit under $\rho_2$.   
Note that this is {\em not} equivalent to the usual notion of semi-conjugacy from
topological dynamics.   We suggest instead the term {\em weakly conjugate}.  Indeed, 
the condition given above is equivalent to the condition that any continuous, 
conjugacy-invariant, real-valued map $f$ on
$\Hom(\Gamma,\HOS)$, satisfies $f(\rho_1)=f(\rho_2)$, and the quotient of
$\Hom(\Gamma,\HOS)$ by this weak- (or, following Ghys, semi-) conjugacy relation is the analog of the {\em character
variety} for linear representations of $\Gamma$.  See \cite{RigidGeom} for a full
discussion.   However, as ``weakly conjugate'' is not yet standard, we defer to tradition and 
use the term semi-conjugacy in the remainder of the paper.  
In our situation, the standard action of $\Aut_+(\Sigma_g)$ is {\em minimal},
and applying the Denjoy trick produces a non-minimal, hence non-conjugate
action. Note, however, that any two minimal, semi-conjugate actions are in
fact conjugate.

As a consequence of Theorem~\ref{thm:main}, we can quickly recover and extend results of Parwani~\cite{Parwani}
and of Farb and Franks~\cite{FarbFranks} concerning the nonexistence of actions of higher regularity.  

\begin{corollary} \label{cor:Parwani}
  For $g \geq 3$, any action of $\Aut_+(\Sigma_g)$ on $S^1$ by
  $C^1$ diffeomorphisms is trivial.
\end{corollary}

Parwani proved this statement under the additional assumption of genus at least 6, however his proof also applies to surfaces with boundary.  Farb and Franks prove the statement for $g \geq 3$, but with $C^1$ replaced by $C^2$.  They also are able to argue for a larger class of groups. 
We give the proof of Corollary~\ref{cor:Parwani} in Section~\ref{sec:conclude}.


\section{Proof of Theorem \ref{thm:main}}

Let $g \geq 2$, and let $\rho: \Aut_+(\Gamma_g) \to \Homeo_+(S^1)$ be a
representation.
The strategy of the proof is to constrain the possible behavior of the
restriction of $\rho$ to the surface subgroup
$\Gamma_g \cong \Inn(\Gamma_g) \subset \Aut_+(\Gamma_g)$.
We will use the notation $\Aut_+(\Gamma_g)$ and $\Mod_{g,1}$ interchangeably
throughout the proof, depending on whether we prefer to evoke $\Gamma_g$ as
an abstract group, or whether it is useful to remember the topology of the
surface~$\Sigma_g$.

\subsection{Action of the surface subgroup} \label{sec:euler}

Before embarking on the proof,
we briefly recall some standard material on rotation numbers
and on the Euler number of a representation. The group $\HOS$ fits in the
exact sequence
\[ 0\to\Z\to\HOZ\to\HOS\to 1, \]
where $\HOZ$
is the group of (orientation-preserving) homeomorphisms of $\R$ which
commute with integer translations.
The {\em translation number of} an element $f\in\HOZ$ is defined by
$\rotild(f)=\lim_{n\to+\infty}\frac{f^n(x)}{n}$.  This limit exists and is independent of the choice of $x\in\R$.
For $f\in\HOS$, the (Poincar\'e) {\em rotation number} of $f$, denoted
$\rot(f)$, is defined as the translation number of any of its lifts
to $\HOZ$ modulo~$\Z$.  
It is easily verified that $\rot$ is conjugacy invariant, and satisfies $\rot(f^{-1}) = - \rot(f)$. 

Any representation $\varphi: \Gamma_g \to \HOS$ can be assigned an
integral {\em Euler number} as follows:
associated to $\varphi$, there is an $S^1$-bundle over $\Sigma_g$ given by 
the quotient of $\widetilde{\Sigma_g}\times S^1$ by the diagonal action of
$\Gamma_g$ \textsl{via}
deck transformations on $\widetilde{\Sigma_g}$ and $\varphi$ on $S^1$. 
The Euler number $\eu(\varphi)$ is the pairing of the Euler class of this
bundle with the fundamental class of $\Sigma_g$; it is (classically) the
obstruction to a section of the bundle.   
Following the interpretation of Milnor and Wood~\cite{Milnor,Wood} the Euler
number of $\varphi$ can also be seen as the obstruction to lifting $\varphi$
to a representation into $\HOZ$. This can be computed by understanding
rotation numbers of individual elements of $\varphi(\Gamma_g)$, or more
precisely, through the {\em translation cocycle},
$\tau\colon\HOS \times \HOS\to\R$ defined by
$(f,g)\mapsto\rotild(\tilde{f}\tilde{g})-\rotild(\tilde{f})-
\rotild(\tilde{g})$.
That this cocycle takes values in $[-1,1]$ is the key to the {\em Milnor-Wood
inequality}, which states that $|\eu(\varphi)|\leq 2g-2$ for any
representation $\varphi\colon\Gamma_g\to\HOS$.

One way to compute the Euler number of an action $\varphi$ of $\Gamma_g$ using rotation numbers is to
decompose $\Sigma_g$ into pairs of pants, and then sum the contribution to the Euler number coming from each pant.   Suppose $P \subset \Sigma_g$ is an embedded pants subsurface with fundamental group
$\langle a,b,c\mid abc=1\rangle$ where $a,b$ and $c$ are freely homotopic
to the boundary curves of $P$.  The surface $P$ inherits an orientation from $\Sigma_g$, and we require $a, b$ and $c$ to have the induced boundary orientation. 
We define $\eu(\varphi_{|P}):=\tau(\varphi(a),\varphi(b))$; this is invariant under cyclic permutations of $a,b,c$.  
Then the Euler number of $\varphi$ is simply the sum $\sum_{P \in \mathcal{P}} \eu(\varphi_{|P})$, where $\mathcal{P}$ is any pants decomposition.  
We refer the reader to~\cite{Ghys01} for a general introduction to Euler
classes. The definition given here that applies to pants subsurfaces
is a contribution of Burger, Iozzi and Wienhard (see~\cite{BIW}, formula~(4.4)
and~Th.~4.6).
The reader may also refer to~\cite[Section~5]{AuMoinsG} for a brief
explanation and proof of well-definedness that avoids the use of bounded
cohomology. 

While the discussion above was general, we now return to our assumption that
$\rho$ is a representation of $\Aut_+(\Gamma_g)$ in $\Homeo_+(S^1)$.
Using a technical result from work of the authors in \cite{RigidGeom},
we prove a first lemma.

\begin{lemma} \label{lem:rot0}
If $a$ is a nonseparating simple closed curve, then $\rot(\rho(a)) =0$.
\end{lemma}

\begin{proof}
It suffices to prove this for a single nonseparating simple closed curve,
since $\Mod_{g,1}$ acts transitively on such curves, and so every nonsimple
closed curve has the same rotation number (up to multiplying by $-1$ to
account for orientation).
Let $a$ be a nonseparating simple closed curve. Since there is an involution
of $\Sigma_g$ that maps $a$ to its inverse, it follows that
$\rot(\rho(a)) = \rot(\rho(a^{-1})) = - \rot(\rho(a))$, so $\rot(\rho(a))$
is either $\frac{1}{2}$ or 0.

Now suppose for contradiction that this common rotation
number for nonseparating simple curves is $\frac{1}{2}$, and let $a,b$ be
standard generators of the fundamental group of
a subsurface $T$ of $\Sigma_g$ homeomorphic to a one-holed torus.
Proposition~5.5 from~\cite{RigidGeom} gives a procedure to produce a 
simple closed curve $c$ contained in $T$ such that
$0 \leq |\rot(\rho(c))| < \frac{1}{2}$.
But this contradicts the fact that $\rot(\rho(a)) = \pm \rot(\rho(c))$.
\end{proof}

As a consequence, we have the following.  

\begin{corollary}
The Euler number of $\rho|_{\Gamma_g}$ is either 0 or $\pm(2g-2)$.
\end{corollary}

\begin{proof}
Let $P_1, P_2, \ldots, P_{2g-2}$ be an oriented pants decomposition of
$\Sigma_g$ such that each boundary curve of each $P_i$ is nonseparating.  
For $i=1,\ldots,2g-2$ let $a_i,b_i,c_i$ be generators of the fundamental group
of $P_i$ with orientation inherited from $P_i$.  
(This is a slight abuse of notation, as
we base all these curves at the same point in~$\Sigma_g$.)

By Lemma~\ref{lem:rot0} and the above mentioned bound on $\tau$, 
for each $i$ we have that $\eu(\rho_{|P_i})$ is either $0$ or $\pm 1$.
Since the boundary curves of $P_i$ are nonseparating, for each $i$ there is an
orientation-preserving homeomorphism $f_i$ of $\Sigma_g$ sending $P_i$ to $P_1$.  
It follows that the triples $(\rho(a_i),\rho(b_i),\rho(c_i))$ are all
conjugate in $\HOS$; with $(\rho(a_i),\rho(b_i),\rho(c_i))$ conjugate to $(\rho(a_1),\rho(b_1),\rho(c_1))$ by the image of the mapping class of $f_i$ under $\rho$.   
Hence the contributions of all $P_i$ to the Euler number $\eu(\rho_{|\Gamma_g})$ are all equal, and so their sum is
either $0$ or $\pm(2g-2)$.
\end{proof}
To treat the case where $\eu(\rho) = \pm(2g-2)$ we use the following theorem of Matsumoto.  

\begin{theorem}[Matsumoto \cite{Matsumoto}]
  If $\Gamma_g$ acts on $S^1$ with Euler number equal to $\pm(2g-2)$,
  then the action is semi-conjugate to the boundary action given by
  embedding $\Gamma_g$ in $\PSL$ as a cocompact lattice.
\end{theorem}
 
The proof of this uses the translation cocycle defined above; a strategy
for a more elementary approach to the proof can be found in \cite{MatsBP}.
 
\begin{lemma} \label{lem:Matsumoto}
  If the Euler number of the restriction of $\rho$ to $\Gamma_g$ is nonzero,
  then up to reversing orientation of the circle, $\rho$ is semi-conjugate
  to the standard action.
\end{lemma} 

\begin{proof}
Suppose that the Euler number of the restriction of $\rho$ to $\Gamma_g$ is
nonzero. Then $\rho(\Gamma_g)$ does not have a finite orbit, and so there
is a canonical {\em minimal set} for the action
(see \textsl{e.g.} \cite[Prop. 5.6]{Ghys01}).
This is the unique closed,
$\Gamma_g$-invariant set contained in the closure of any orbit, on which
$\Gamma_g$ acts minimally.
Let $K \subset S^1$ denote the minimal set.  It is equal to $S^1$ 
if the action is minimal, and homeomorphic to a Cantor set if
not.

Since $\Gamma_g$ is normal in $\Aut_+(\Gamma_g)$, the action of
$\Aut_+(\Gamma_g)$ on $S^1$ preserves $K$.
Thus, $K$ is a closed, invariant set on which $\Aut_+(\Gamma_g)$ acts
minimally, hence is the minimal set for $\rho$.    
Up to semi-conjugacy, we may in fact assume that the action of $\rho$ is
minimal. Precisely, if $K$ happens to be a Cantor set, then let
$h\colon S^1 \to S^1$ be a continuous, surjective map that is injective on
$K$ and collapses each of its complementary intervals to a point.
The action of $\Gamma_g$ descends naturally to an action on $h(S^1)$ which is
minimal and semi-conjugate to the original action.
Thus, going forward, we assume that $\rho(\Aut_+(\Gamma_g))$ acts minimally.
As noted above (see also \cite{Ghys01}) minimality implies that the action
of $\rho(\Gamma_g)$ is {\em conjugate} to the standard boundary action.
We claim that this is enough to determine the action of $\Aut_+(\Gamma_g)$
up to conjugacy.  Indeed, minimality of the action of $\Gamma_g$ implies that the set of attracting fixed points of hyperbolic
elements of $\Gamma_g$ represented by closed curves in $\Sigma_g$ is dense
in $S^1$.  If $x \in S^1$ is the attractor of some $\gamma \in \Gamma_g$ represented by a closed curve, then each $\varphi \in \Aut_+(\Gamma_g)$ must necessarily have $\rho(\varphi)$ map $x$ to the (unique) attractor of $\varphi(\gamma) \in \Gamma_g$.
This determines the action of $\rho(\varphi)$ on a dense set, hence completely
specifies it as a homeomorphism.
\end{proof} 

The remainder of the proof of Theorem \ref{thm:main} consists of showing that
the Euler number of $\rho|_{\Gamma_g}$ is 0 only when $\rho$ 
factors through the abelianization of $\Aut_+(\Gamma_g)$.

\subsection{Orbifold groups and their Euler numbers}   \label{sec:orbifold}
Our motivation for the remainder of the proof comes from the following observation.  

\begin{observation} \label{obs:normalizer}
Fix an embedding of $\Gamma_g$ as a lattice in $\PSL$. If $\Delta$ is a
Fuchsian group that normalizes $\Gamma_g$, then $\Delta$ embeds faithfully
into $\Aut_+(\Gamma_g)$.
\end{observation}

\begin{proof}
This is just the observation that the centralizer of $\Gamma_g$ in $\PSL$
is trivial, because $\Gamma_g$ is nonelementary.
\end{proof}

Thus, to get our hands on concrete elements of $\Aut_+(\Gamma_g)$, we produce embeddings of $\Gamma_g$ into $\PSL$ with large normalizers.  These are constructed by realizing $\Sigma_g$ as a regular cover of a hyperbolic orbifold.  
To this end, we recall a few facts about cocompact Fuchsian groups;
the reader may refer to~\cite{Katok} for a general introduction.
Any cocompact Fuchsian group  $\Gamma$ has a {\em signature}, of the form
$(g;m_1,\ldots,m_r)$, corresponding to the presentation
\[\Gamma = \langle a_1,b_1,\ldots,a_g,b_g,q_1,\ldots,q_r\mid
q_1\cdots q_r\cdot [a_1,b_1]\cdots [a_g,b_g]=q_1^{m_1}=\cdots=q_r^{m_r}=1
\rangle. \]
Its covolume (the volume of the quotient $\H^2/\Gamma$) is given by the
formula $\mu(\Gamma)=2\pi \left(2g-2+\sum_{i=1}^r (1-\tfrac{1}{m_i})\right)$;
and its \emph{orbifold Euler characteristic} $\chi(\Gamma)$ is defined to be
$-\mu(\Gamma)/2\pi$.
When $r=0$ this is the fundamental group of a closed surface, with the usual
definition of Euler characteristic and hyperbolic volume. When $r>1$, such
a group $\Gamma$ should be thought of as the holonomy representation of a
hyperbolic surface with $r$ cone points, of cone angles $\frac{2\pi}{m_i}$.
A $(g;m_1,\ldots,m_r)$ {\em orbifold} is simply the quotient of $\H^2$ by
such a group $\Gamma$.

The definition of the {\em Euler number} of a representation $\Gamma_g\to\HOS$ discussed above 
can be extended to representations of any orbifold group.  
As in the case of representations of surface groups, one can form the quotient of $\H^2 \times S^1$ 
by the diagonal action of $\Gamma$; the result is not generally an $S^1$-bundle, but rather a Seifert fibered space. 
Seifert fibered spaces have associated Euler numbers; analogous to the circle bundle case, the Euler number can also be thought of as the obstruction to a $\Gamma$-invariant section of the projection $\H^2 \times S^1 \to \H^2$.  
An equivalent definition can be obtained by thinking of the cone points as
topological boundary components and using the definition 
from~\cite{BIW} of Euler number for representations of surfaces with boundary,
as was used in our pants-decomposition definition in Section~\ref{sec:euler}.

Both the orbifold Euler characteristic and the Euler number are multiplicative under covers.  If $\Delta \subset \Gamma$ is a finite index subgroup of index $k$, then $\chi(\Delta) = k\chi(\Gamma)$ (see eg~\cite[Th.~3.1.2]{Katok}) and for any representation $\rho: \Gamma \to \Homeo_+(S^1)$, we also have $\eu(\rho|_\Delta) = k\,\eu(\rho)$.  
The following observation follows quickly from the definition (most easily from that given in \cite{BIW}) and was used by Calegari in~\cite{CalegariForcing}.

\begin{observation}
Let $\Gamma$ have signature $(g;m_1,\ldots,m_r)$, and let $\rho: \Gamma \to \HOS$ be a representation.  Then there exists $m \in \Z$ such that 
\[ \eu(\rho) = m + \sum_{i=1}^r \rot(\rho(q_i))). \]
\end{observation}

As such, there are situations when one can easily guarantee that the Euler
number of a representation is nonzero.
There are two specific examples of this which we will use in the sequel.

\begin{example} \label{ex:2222g}
The $(0; 2, 2, 2, 2g)$ orbifold group has presentation
\[ \langle a, b, c, d \mid a^2, b^2,  c^2,  d^{2g}, abcd\rangle \]
and has Euler characteristic $\frac{1-g}{2g}$.  
If this group acts on the circle by homeomorphisms, and the action of $d^2$ is nontrivial, 
then the Euler number of the action is of the form $k/2 + \rot(d)$.  As $d^2$ acts nontrivially and $d$ is finite order, we conclude that $\rot(d) \notin \{\frac{1}{2}, 0\}$, so the Euler number of the action is nonzero.
\end{example}

\begin{example} \label{ex:334}
The $(0;3,3,4)$ orbifold group has presentation 
\[ \langle a,b,c \mid a^3, b^3, c^4, abc\rangle, \]
and has Euler characteristic $\frac{-1}{12}$.
If it acts on the circle by homeomorphisms and the action of $c$ is nontrivial, then the Euler number of the action is of the form 
$k/3 + m/4$ for integers $k, m$ with $m \neq 0$~$\mathrm{mod}~4$,
hence it is nonzero.
\end{example} 

We conclude these preliminaries with a final (and well-known) ingredient
for our proof.

\begin{lemma}  \label{lem:torsionfree}
  Let $\Gamma$ be an orbifold group of signature $(g;m_1,\ldots,m_r)$,
  with the standard presentation given above, and let
  $\varphi\colon\Gamma \to G$ be a surjective homomorphism to a finite group.
  Suppose that each finite order generator $q_i$ of $\Gamma$ is mapped to an
  element of $G$ of order $m_i$.
  Then $\ker(\varphi)$ is the fundamental group of a compact surface of
  genus $g$ given by the formula $2-2g = \chi(\Gamma)\times|G|$.
\end{lemma}

\begin{proof}
  It is a classical standard fact that
  all finite order (\textsl{i.e.}, elliptic) elements of $\Gamma$ are
  conjugate to some power of one of its finite order standard generators
  (see eg~\cite[Th.~3.5.2]{Katok}). Hence, it
  follows from the asumption that $\ker(\varphi)$ is torsion free.
  As $\Gamma$ is cocompact, and $\ker(\varphi)$ is finite index
  (of index $|G|$), the group $\ker(\varphi)$ is cocompact as well,
  hence it is a surface group.
  The genus calculation follows from multiplicativity of Euler characteristic
  under covers.
\end{proof}

\subsection{Finishing the proof}
We now apply the framework above to our situation, finding normal genus $g$
surface subgroups inside of the orbifold groups given in
Examples~\ref{ex:2222g} and~\ref{ex:334}, and use this to conclude our proof.

Consider first the group given in Example~\ref{ex:2222g}, which has an
(equivalent) presentation
$\langle a, b, c \mid a^2, b^2, c^2, (abc)^{2g} \rangle$.   
Define a surjective homomorphism from this group to the dihedral group
$\langle r, s \mid r^{2g}, s^2,  srsr\rangle$ of order $4g$ by
\begin{align*}
& a \mapsto r^g, \\
& b \mapsto sr \\ 
& c \mapsto sr^{2-g},
\end{align*}
so $abc \mapsto r$.
Since the standard, finite order generators $a, b, c$ and $d := (abc)^{-1}$
map to elements of their respective orders, Lemma~\ref{lem:torsionfree}
states that the kernel $K$ of this morphism is a torsion free subgroup of
index $4g$, and the Euler characteristic of the regular cover corresponding
to the kernel is $4g(\frac{1-g}{2g}) = 2-2g$. Thus, we obtain $\Sigma_g$ as
a regular cover of the $(0; 2, 2, 2, 2g)$ orbifold.

Recall that $\rho$ is assumed to be an action of $\Aut_+(\Gamma_g)$
on $S^1$. By Observation~\ref{obs:normalizer}, the $(0; 2, 2, 2, 2g)$
orbifold group embeds in $\Aut_+(\Gamma_g)$, with $\Inn_g \cong \Gamma_g$
agreeing with the kernel $K$ of the homomorphism defined above.
It follows from Example~\ref{ex:2222g} and multiplicativity of the Euler
number that if $\rho(d)$ has order greater than 2, then restriction of $\rho$
to $K$ also has nonzero Euler number. Thus, by Lemma~\ref{lem:rot0}, the
Euler number of the restriction of $\rho$ this subgroup equals $\pm(2g-2)$,
so by Lemma~\ref{lem:Matsumoto} $\rho$ agrees with the standard action.

Thus, we have proved Theorem~\ref{thm:main} under the additional hypothesis
that $\rho(d)^2 \neq \id$. To remove this hypothesis, we use recent work of
Lanier and Margalit on normally generating mapping class groups.

\begin{theorem}[Lanier--Margalit \cite{LM}] \label{thm:LM}
Let $g \geq 2$. Then every nontrivial, periodic mapping class that is not a
hyperelliptic involution normally generates the commutator subgroup of $\Mod_g$.
\end{theorem}
Recall that the abelianization of $\Mod_g$ is trivial if $g\geq 3$, and is
$\Z/10\Z$ if $g=2$ (see \textsl{e.g.}~\cite{Korkmaz}); hence these
periodic mapping classes normally generate $\Mod_g$ if~$g\geq 3$.

The key step of Lanier--Margalit's proof is as follows. Given any such
periodic mapping class $f$, they find simple closed curves $\alpha$ and
$\beta$ that intersect once, such that the product of Dehn twists
$\tau_{\alpha}\tau_{\beta}^{-1}$ lies in the normal closure of $f$.
(See~\cite[Lemma~2.3]{LM}.)
This step can be carried out in exactly the same way in the group
$\Mod_{g,1}$.  
Since such elements $\tau_{\alpha}\tau_{\beta}^{-1}$
also generate the commutator subgroup of $\Mod_{g,1}$,
the proof goes through verbatim
and the conclusion of Theorem~\ref{thm:LM} holds in this case as well. 

Using this result, we may now quickly conclude our proof in the case of
genus $g \geq 3$. In this case, the element $d^2$ has order $g \geq 3$
in $\Mod_{g,1} \cong \Aut_+(\Gamma_g)$,
so is not a hyperelliptic involution. Thus, if $\rho(d)^2$ is trivial,
the normal closure of $d^2$ is in the kernel of $\rho$ as well, so by
Theorem~\ref{thm:LM} $\rho$ is trivial.

For the case of genus $2$, the element $d^2$ \emph{is} the hyperelliptic
involution of $\Sigma_2$, so the argument above does not immediately apply.
So we work instead with the group
$\langle a,b \mid a^3, b^3, (ab)^4 \rangle$
from Example~\ref{ex:334}, following the same strategy.  
Define a homomorphism $\varphi$ from this group
to the finite group $\SL_2(\F_3)$ by
$\varphi(a) = \left(\begin{smallmatrix}1 & 1 \\ 0 & 1\end{smallmatrix}\right)$; $\varphi(b) = \left(\begin{smallmatrix}1 & 0 \\ 1 & 1\end{smallmatrix}\right)$.
This morphism is easily seen to be surjective, as $\varphi(a)$ and $\varphi(b)$
are images of standard generators of $\SL_2(\Z)$ under the natural map to
$\SL_2(\Z/3\Z)$. Lemma~\ref{lem:torsionfree} again implies that the kernel
of $\varphi$ is a torsion free subgroup, hence the fundamental group of a
surface. Since the Euler characteristic of the $(0; 3,3,4)$ orbifold is
$\frac{-1}{12}$, and $|\SL_2(\F_3)| = 24$, this surface has genus 2, so by
Observation~\ref{obs:normalizer}, we can identify the $(0; 3,3,4)$ group
with a subgroup of~$\Aut_+(\Gamma_2)$.

Example~\ref{ex:334} now implies that, if $\rho\colon \Aut_+(\Gamma_2)\to\HOS$
is such that the Euler number of the restriction to the $(0; 3,3,4)$ group
is zero, then $\rho(ab) = \id$.  Since $ab$ has order 4, it is not the
hyperelliptic involution, so Theorem~\ref{thm:LM} implies that
kernel of $\rho$ contains the commutator subgroup of $\Aut_+(\Gamma_2)$,
hence $\rho$ factors through its abelianization, $\Z/10\Z$.
This completes the proof.
\qed


\section{Concluding remarks}  \label{sec:conclude}

\begin{proof}[Proof of Corollary \ref{cor:Parwani}]
Let $\rho: \Aut_+(\Sigma_g) \to \Diff^1_+(S^1)$, where $g \geq 3$.
By Theorem~\ref{thm:main}, $\rho$ is either trivial or is semi-conjugate
to the standard action (up to reversing orientation).

Suppose for contradiction that $\rho$ is a $C^1$ action that is semi-conjugate to the
standard action.
Let $\gamma \in \Gamma_g$ be an element represented by a separating simple
closed curve $c$ on $\Sigma_g$, so that one connected component of
$\Sigma \smallsetminus c$ has genus $h \geq 2$. Then the stabilizer of
$\gamma$ in $\Aut_+(\Sigma_g)$ contains a copy of $\Mod_{h}^1$, the
mapping class group of the genus $h$ surface with one boundary component.
Since the standard boundary action restricted to this subgroup has a global
fixed point, and since the property of having a global fixed point is
preserved under semi-conjugacy, it follows that $\rho(\Mod_{h}^1)$ acts on
$S^1$ with a fixed point. Let $x$ be a point on the boundary of the fixed
set of $\rho(\Mod_{h}^1)$.
Since $\Mod_{h}^1$ has trivial abelianization (see e.g. \cite{Korkmaz}), the linear representation obtained by taking derivatives at $x$ is trivial.    
However, the {\em Thurston stability theorem} \cite{ThurstonStability} states that, if $G$ is any finitely generated, nontrivial group of germs of $C^1$ diffeomorphisms at a point of a manifold, with trivial linear part, then $G$ admits a surjective morphism to $\Z$.  It follows that $\Mod_{h}^1$ (via its image in the group of germs of diffeomorphisms at $x$) has a nontrivial morphism to $\Z$, contradicting the fact that its abelianization is trivial.  
\end{proof}

\begin{remark}
Our proof of Theorem \ref{thm:main} relied heavily on torsion elements, so does not generalize to finite index subgroups of $\Aut_+(\Gamma_g)$.
However, quotients of finite index subgroups of $\Mod_{g,1}$ are not well
understood, so one does not expect an analogous result to follow along the
same lines. For example, it is a long standing question -- or perhaps conjecture -- of Ivanov~\cite{Ivanov} whether all finite index subgroups of mapping
class groups have trivial abelianization. (See~\cite{Margalit} for a discussion on the current status of the problem.)
There are many non semi-conjugate actions of $\Z^d$ on the circle,
for any $d \geq 1$; for example, one may take representations into $\SO(2)$, then their (semi)-conjuagcy classes 
are distinguished by rotation numbers.  Thus, any subgroup $\Gamma \subset \Aut_+(\Gamma_g)$ with
$H^1(\Gamma, \Z) \neq 0$ would have many nontrivial and non-semi-conjugate
actions on the circle. 
\end{remark}

Given the remark above the relevant remaining question is as follows. 

\begin{question}
Is every faithful action of a finite-index subgroup of $\Aut_+(\Gamma_g)$
on $S^1$  semi-conjugate to the
standard action? 
\end{question}

We hope to address this in future work.

\begin{acknowledge}
The authors thank the Universidad de la Rep\'ublica in Montevideo, Uruguay,
for the hospitality during the workshop on Groups, Geometry and Dynamics
in July 2018.  We are grateful to B. Farb and J. Lanier for helpful comments, and 
Lanier for explaining that the proof of Theorem~\ref{thm:LM} generalizes to surfaces with a marked point.
K.~Mann was partially supported by NSF grant DMS-1606254.
\end{acknowledge}

\bibliographystyle{plain}

\bibliography{Biblio}

\end{document}